\documentclass[11pt]{amsart}
\usepackage{amsmath, amsthm, amssymb, latexsym,url}
\usepackage{hyperref}

\author{Joseph Vandehey}
\thanks{Email: \href{mailto:vandehey.1@osu.edu}{\nolinkurl{vandehey.1@osu.edu}}}
\title{Differencing methods for Korobov-type exponential sums}
\date{\today}

\newtheorem{thm}{Theorem}[section]
\newtheorem{cor}[thm]{Corollary}
\newtheorem{rem}[thm]{Remark}
\newtheorem{lem}[thm]{Lemma}

\newtheorem{prop}[thm]{Proposition}
\newtheorem{claim}[thm]{Claim}

\newcommand{\ord}{\operatorname{ord}}

\allowdisplaybreaks

\begin{document}

\begin{abstract}
We study exponential sums of the form $\sum_{n=1}^N e^{2\pi i a b^n/m}$ for non-zero integers $a,b,m$. Classically, non-trivial bounds were known for $N\ge \sqrt{m}$ by Korobov, and this range has been extended significantly by Bourgain as a result of his and others' work on the sum-product phenomenon. We use a new technique, similar to the Weyl-van der Corput method of differencing, to give more explicit bounds bounds that become non-trivial around the time when $\exp(\log m/\log_2\log m) \le N$. We include applications to the digits of rational numbers and constructions of normal numbers.
\end{abstract}

\maketitle

\section{Introduction}

In this paper, we will be interested in bounding exponential sums of the form
\begin{equation}\label{eq:korobovtype}
\sum_{n=1}^N e\left( \frac{a}{m} b^n \right),
\end{equation}
where $a,b,m$ are integers, $b\ge 2$, $m\ge 1$, and $e(z) = e^{2\pi i z}$.  The summands of \eqref{eq:korobovtype} are periodic with period equal to $\ord(b,m)$, the multiplicative order of $b$ modulo $m$. These sums were most famously studied by Nikolai Mikhailovich Korobov, who called them ``rational exponential sums containing an exponential function" \cite{KorobovBook}; however, in the interest of brevity, we shall call them Korobov-type exponential sums. The applications of such sums have included the study of the digits in rational numbers (seemingly a favorite topic of Korobov), the construction of normal numbers, and the study of the Diffie-Hellman key exchange. A fuller treatment of the history and applications of these sums can be found in \cite{KS}.

Korobov's initial results on these types of exponential sums said that \eqref{eq:korobovtype} should be no larger than $\sqrt{m}(1+\log m)$ provided $\gcd(a,m)=1$ and $N\le \ord(b,m)$. (We will include the explicit statement of this and other results in a later section.) This is non-trivial---i.e., better than the trivial bound of $N$---if $ \sqrt{m}(1+\log m)\le N$. Korobov had a second result for the restricted case where $m$ is a large power of a fixed prime, and it is non-trivial roughly when $\exp((\log m)^{2/3}) \le N$, a vastly improved range. Very few other results on Korobov-type exponential sums appear to be in the literature. Much of the recent interest in the problem has come as a result of bounds on the sum-product phenomenon over finite fields, including work by Bourgain, Chang, and Garaev, among others. We include the most general result of this type later on, which appears to be non-trivial approximately  when $\exp(\log m/ \sqrt{ \log \log m}) \le N$, although it only gives strong bounds when $\exp(\log m / \sqrt{\log \log \log m})\le N$. While we are interested in the case of general $m$ in this paper, we note that the case where $m$ is a prime is also well studied, see for example \cite{Kerr} and Chapter 4 of \cite{KS}.

Much of the study of exponential sums $\sum e(f(n))$ for arbitrary $f(n)$ has been based on the methods of Weyl, van der Corput, and Vinogradov. See \cite[Ch.~8]{IK} for a basic treatment of this material. One of the classic techniques in this study has been Weyl differencing, sometimes referred to as Weyl-van der Corput differencing. This seeks to bound the sum $\sum e(f(n))$ by sums of the form $\sum e(\Delta_\ell f(n))$ where $\Delta_\ell f(n) = f(n+\ell)-f(n)$. For polynomials and many related functions, this $\Delta_\ell$ operator reduces the complexity of the resulting function and allows for stronger estimations to be applied. However, classically this method has not been very effective for studying Korobov-type exponential sums, as the $\Delta_\ell$ operator leaves exponential functions as exponential functions still. (It should be noted Korobov has successfully adapted the methods of Vinogradov to Korobov-type exponential sums, resulting in Theorem \ref{thm:Korobovprimesum}.)

In this paper, we will develop a new variant of Weyl differencing designed to bound Korobov-type exponential sums $\sum e(a b^n /m)$ by sums of the form $\sum e(a' b^n /m')$ with $m'$ significantly smaller than $m$. We will use this new result to prove the following bounds on exponential sums of Korobov type. In the following results, we will often let $P$ be some finite set of primes and denote by $\mathbb{N}_P$ the set of natural numbers whose prime factors all belong to $P$. We let $1\in \mathbb{N}_P$ as well.

\begin{thm}\label{thm:main}
Let $P$ be some finite set of primes. 
Let $m\in \mathbb{N}_P$, let $a\in \mathbb{Z}$ with $\gcd(a,m)=1$, and let $b\ge 2$ be an integer with $\gcd(m,b)=1$. Let $k\in \mathbb{N}_{\ge 0}$. Then for all integers $N\ge 1$, we have that
\begin{align*}
\left| \sum_{n=1}^N e\left( \frac{a}{m} b^n\right)\right| &\le  K_1 \cdot K_2^k \cdot m^{\frac{1}{2^{k+2}-2}} N^{1-\frac{k+3+c_k}{2^{k+2}}}(1+\log m)^{2^{-k}}\\ &\qquad  +K_3\cdot m^{-\frac{1}{2^{k+2}-2}} N^{1+\frac{k-1+c_k}{2^{k+2}}} (1+\log m)^{2^{-k}}.
\end{align*}
where $K_1, K_2, K_3$ are fixed quantities depending only on $P$ and given explicitly in \eqref{eq:K1def}--\eqref{eq:K3def} and 
\[
c_k = c+O\left( \frac{k+7}{2^{k-1}}\right)
\]
with $c$ some fixed constant (not dependent on any other quantities) given explicitly in \eqref{eq:cexplicit} and the implicit constant here is $1$.
\end{thm}

Theorem \ref{thm:main} follows from Theorem \ref{thm:recursive}, which we will state later and which gives all of the constants explicitly via recursive formulas. 
 By carefully analyzing the recursive formulas in Theorem \ref{thm:recursive}, we obtain the following corollaries.

\begin{cor}\label{cor:first}
Let $P$ be a finite set of primes.
Let $m\in \mathbb{N}_P$, let $a\in \mathbb{Z}$ with $\gcd(a,m)=1$, and let $b\ge 2$ be an integer with $\gcd(m,b)=1$. Suppose $\epsilon>0$. Then there exists $\delta=\delta(\epsilon)>0$ and $C=C(\epsilon,b,P)>0$ such that
\[
\left| \sum_{n=1}^N e\left( \frac{a}{m} b^n\right)\right| \le C m^{-\delta} N(1+\log m)
\]
for $N\ge m^{\epsilon}$.
\end{cor}

\begin{cor}\label{cor:second}
Let $P$ be a finite set of primes. Suppose $m\in \mathbb{N}_P$, $a\in \mathbb{Z}$ with $\gcd(a,m)=1$, and $b\ge 2$ be an integer with $\gcd(m,b)=1$
Let $c$ be as in Theorem \ref{thm:main} and $k\ge 1$. If $m^{1/(k+c+2)}\le N \le m^{1/(k+c+1)}$, then 
\begin{equation}\label{eq:corsecondeq}
\left| \sum_{n=1}^N e\left( \frac{a}{m} b^n\right)\right| \le C_1 \cdot C_2^{k} \cdot N^{1-\frac{1}{2^{k+3}}}(1+\log m)^{2^{-k}},
\end{equation}
where $C_1=C_1(b,P)>0$ and $C_2=C_2(b,P)>0$. 

For any $\epsilon>0$, there exists $m_0=m_0(\epsilon)$ such that if $m\ge m_0$, then the bound given in \eqref{eq:corsecondeq} will be non-trivial if $k \le \log_2 \log m - (2+\epsilon)\log_2\log \log m$.
\end{cor}

We may combine the two corollaries to produce a strong combined theorem, showing that the results of this paper give bounds on Korobov-type exponential sums that are non-trivial much faster than the previous results of Bourgain and others, although still not quite as fast as the results of Korobov for prime power denominators.

\begin{thm}\label{thm:secondary}
Let $P$ be a finite set of primes and let $b\ge2$ be an integer relatively prime to every prime in $P$. Then for all sufficiently large $m\in \mathbb{N}_P$, for all $a\in \mathbb{Z}$ with $\gcd(a,m)=1$, and all $N\in \mathbb{N}$ satisfying
\[
\exp\left( \frac{\log m}{\log_2 \log m - 3\log_2\log \log m}\right)\le N,
\] 
we have that
\[
\frac{1}{N}\left| \sum_{n=1}^N e\left( \frac{a}{m} b^n\right)\right| = \exp\left( -c(\log \log m)^{3/2}\right),
\]
for some small $c>0$ that depends only on $P$.\footnote{The constant $c$ here has no relation to the constant $c$ in Theorem \ref{thm:main}.}
\end{thm}

This paper is arranged as follows. In Section \ref{sec:priorresults}, we shall state many of the results we need from previous papers, as well as some related results on Korobov-type exponential sums for comparison. In Section \ref{sec:differencinglemma}, we will give our version of the Weyl differencing lemma. In Section \ref{sec:iteration}, we apply the differencing lemma to show that if one class of bounds hold on Korobov-type exponential sums, then a different class of bounds must also hold. In Section \ref{sec:recursive}, we use the result of the previous section iteratively to prove Theorem \ref{thm:recursive}, and then perform asymptotic analysis on the various constants that appear in order to prove Theorem \ref{thm:main}. In Section \ref{sec:effectivity}, we then perform further asymptotic analysis on when the results of Theorem \ref{thm:recursive} are non-trivial in order to prove Corollaries \ref{cor:first} and \ref{cor:second}.  Finally, in section \ref{sec:applications}, we include some applications to other problems. 

We will make use of standard asymptotic notations throughout this paper. By saying $f(x) = O(g(x))$ or, equivalently, $f(x) \ll g(x)$, we mean that there exists some constant $C$ such that $|f(x)| \le C | g(x)|$. By $f(x)=o(g(x))$, we mean that $f(x)/g(x)$ goes to $0$ as $x$ goes to infinity.

Given a prime $p$,  we shall say $p^k ||m $ to mean that $p^k|m$ but $p^{k+1} \nmid m$.

\section{Prior Results}\label{sec:priorresults}

\subsection{On $\ord(b,m)$}\label{sec:onord}

The following result gives $\ord(b,m)$, the multiplicative order of $b$ modulo $m$, explicitly. We record the version of it found in  Bailey and Crandall \cite{BC}, although the result is first stated in \cite{Korobov72}, although the proof is essentially a consequence of Lemma 1 in \cite{Korobov70}

\begin{lem}\label{lem:KorobovMultOrder}
Let $b,m>1$ be coprime integers with $m$ having prime factorization $m=p_1^{t_1} p_2^{t_2}\dots p_s^{t_s}$. Let $\tau_1=\tau_1(m) = \ord(b,p_1p_2\dots p_s)$ and define $\beta_i$ by
\[
p_i^{\beta_i} || b^{(\mu+1)\tau_1}-1,
\]
where $\mu=1$ if $m$ is even and $\tau_1$ is odd and $b\equiv 3 \pmod{4}$; and otherwise $\mu=0$. Let $m_1=m_1(m)$ be defined by
\[
m_1 = p_1^{\min(t_1,\beta_1)} p_2^{\min(t_2,\beta_2)} \dots p_s^{\min(t_s,\beta_s)}.
\]
Let $\tau' =\tau'(m)= 2\tau_1$ if $\mu=1$ and $m\equiv 0 \pmod{4}$; and $\tau'=\tau_1$ otherwise.

Then $\ord(b,m) = \frac{m}{m_1} \tau'$.
\end{lem}

Suppose that $b>1$ and $P=\{p_1,p_2,\dots,p_s\}$ is a finite set of distinct primes that are all relatively prime to $b$. Let $\beta_i'$ be defined by
\[
p_i^{\beta'_i} || b^{2\ord(b,p_1p_2\dots p_s)} -1,
\]
so that for any $m$ whose prime factors come from $P$, we have that the $\beta_i$'s are less than the corresponding $\beta_i'$'s. Therefore, for any such $m$, we have that $m_1(m)$ is bounded by
\begin{equation}\label{eq:Mdef}
 M=M(P):= p_1^{\beta_1'} p_2^{\beta_2'} \dots p_s^{\beta_s'}. 
\end{equation}
Since $M|b^{2 \ord(b,p_1p_2\dots p_s)}-1$ we have that 
\begin{equation}\label{eq:Mbound}
M \le b^{2Q}
\end{equation}
where $Q:= p_1 p_2\dots p_s$.

Since $\tau'\ge 1$, we have that
\begin{equation}\label{eq:ordMbound}
\frac{m}{M} \le \ord(b,m).
\end{equation}

\subsection{Korobov's estimates}

The following result appears to be first due to Korobov.\footnote{The author has seen some discussion that this was a folklore theorem and Korobov was merely the first to put it to paper, but has found no proof of this.} The result is incorrectly stated in both locations cited here, since if $m=m_1$, then the result of the lemma still holds provided $d=1$.

\begin{lem}[Lemma 2 in \cite{Korobov72}, Lemma 32 in \cite{KorobovBook}]\label{lem:shortKorobovsum}
Let $b,m>1$ be coprime integers with $m_1$ defined as in Lemma \ref{lem:KorobovMultOrder}. Let $a$ be an integer and let $d:=\gcd(a,m)$. If $d=1$ or $d< m/m_1$, then for all $N\in [1,\ord(b,m)]$, we have
\begin{equation}\label{eq:shortKorobovsum}
\left| \sum_{n=1}^N e\left( \frac{a}{m} b^n\right)\right| < \sqrt{m/d} (1+\log(m/d)).
\end{equation}
\end{lem}

\begin{lem}\label{lem:longKorobovsum}Let $P$ be a finite set of primes.
Let $m\in \mathbb{N}_P$. For any positive integer $N$ and any $a\in \mathbb{Z}$ with $\gcd(a,m) =1$, we have
\begin{equation}\label{eq:longKorobovsum}
\left| \sum_{n=1}^N e\left( \frac{a}{m} b^n\right) \right| < \left( m^{1/2}  + M m^{-1/2}N\right)(1+ \log m),
\end{equation}
where $M=M(P)$ is defined as in \eqref{eq:Mdef}.
\end{lem}

\begin{proof}
We follow the idea of Bailey and Crandall \cite{BC}.

We note again that $e(ab^n/m)$ is cyclic with period length $\ord(b,m)$. Thus, the sum on the left-hand side of \eqref{eq:longKorobovsum} can be  rewritten as $\lfloor N/\ord(b,m)\rfloor$ sums that range from $n=1$ to $\ord(b,m)$ together with an additional sum (if $N$ is not a multiple of $\ord(b,m)$) of the form \eqref{eq:shortKorobovsum}. Each of these sums has norm strictly less than $\sqrt{m} (1+\log(m))$ by Lemma \ref{lem:shortKorobovsum}, so in total we have
\begin{align*}
\left| \sum_{n=1}^N e\left( \frac{a}{m} b^n\right) \right|  &<\left( 1+\frac{N}{\ord(b,m)}\right) \sqrt{m} (1+\log(m))\\
&=\left( 1+\frac{MN}{m } \right) \sqrt{m}(1+\log m),
\end{align*}
by applying \eqref{eq:ordMbound}.
Simplifying gives the desired bound.
\end{proof}

This next result has the strongest requirements on $m$, namely that it is a power of a fixed odd prime, but it provides non-trivial estimates in much much smaller ranges, namely for $\exp((\log m)^{2/3}) \le N$.

\begin{thm}[Theorem 4 in \cite{Korobov72}, Theorem 33 in \cite{KorobovBook}]\label{thm:Korobovprimesum} Let $p$ be an odd prime.
Let $\gcd(b,p)=1$, $\gcd(a,p)=1$, $\alpha> 16\beta$, where $\beta$ is some number dependent only on $p$ and $b$. Let $r$ be given by $r = (\log p^\alpha)/\log N$. Provided $2\le r\le \alpha/8\beta$, then we have that
\[
\left| \sum_{n=0}^N e\left( \frac{a}{p^\alpha} b^n \right)\right| < 3 N^{1- \gamma/r^2}= 3N\cdot \exp\left(-\frac{\gamma (\log N)^3}{(\log p^\alpha)^2}  \right),
\]
where $\gamma=1/(2\cdot 10^6)$.
\end{thm}

We note that the case where $m=2^k$ has been studied by Shparlinski \cite{KS,Shparlinski96}, although his estimations have been over complete sums where $N =\ord(b,m)$.

Due to the small size of the constant $\gamma$, Corollary \ref{cor:second} will give generally give superior bounds to Theorem \ref{thm:Korobovprimesum} when $k$ is small, but for large $k$, Korobov's estimate here is superior.

\subsection{Bourgain's bounds}

Many of the recent results on Korobov-type exponential sums have come about as a result of the study of the sum-product phenomenon in finite fields. The sum-product phenomenon loosely says that if $A$ is a non-empty subset of a ring, then one cannot have that both the sum set $A+A:= \{x+y:x,y\in A\}$ and the product set $A\cdot A := \{xy: x,y\in A\}$ are simultaneously ``small"---that is, close in size to $A$. This is relevant to Korobov-type exponential sums since if one considers the set $A=\{b^n : 1\le n \le N\}$ over the finite field $\mathbb{Z}_m$, then the product-set should be reasonably small. A highly readable account of the connections between these phenomenon over fields of prime order can be found in Garaev \cite{Garaev} and in Kurlberg \cite{Kurlberg}.

The strongest results on Korobov-type exponential sums are due to Bourgain \cite{Bourgain1}. While we are only citing results from this paper, we should emphasize that these are the product of continual refinements to the methods it contained stretching back over several works. We suggest \cite{BC,Bourgain2,BGK,BKT} and the references therein for those interested in studying these methods further.

We have kept much of the notation as it was in Bourgain's paper, only exchanging the denominator $q$ for the denominator $m$, since that is the convention we use elsewhere in this paper. We note that the terms $m'$ and $m_1$ used in this subsection do not have any relation to the terms $m'$ and $m_1$ used elsewhere in this paper.

\begin{thm}[Theorem 8.28 in \cite{Bourgain1}]\label{thm:Bourgain1}
Let $\gamma \in (0,1)$. Then there exists $\epsilon=\epsilon(\gamma)>0$ such that the following holds. Let $m\in \mathbb{Z}_+$ be  a sufficiently large number and let $a\in \mathbb{Z}_m^*$ (where the star denotes the invertible elements) satisfy
\[
\ord(b,m') > (m')^\gamma \text{ if }m'|m \text{ and }m'>m^{\epsilon}.
\]
Then for $m^\gamma < N \le \ord(b,m)$, we have
\[
\max_{a\in \mathbb{Z}_m^*} \left| \sum_{n=1}^N e\left( \frac{a}{m} b^n\right) \right| < N^{1-\epsilon}.
\]
\end{thm}

This result is stated inexplicitly in Bourgain's paper; however, we can do some calculations, referencing equations from the paper, to estimate when this becomes stronger than the trivial bound of $N$. The ``sufficiently large" condition on $m$ requires that $2 < m^{\gamma\epsilon} $ by equations $(8.8)$ and $(8.28)$. Equations $(8.15)$ and $(8.26)$ imply that $\epsilon \ll \gamma^2 \tau(\gamma)/k(\gamma)$, where $\tau(\gamma)$ and $k(\gamma)$ are defined in Theorem $(**)$. These terms do not appear to be explicitly defined, but are approximately given by
\[
\tau(\gamma) \approx  \frac{\gamma}{\exp \exp (4/\tau(\gamma/2,\gamma))} \qquad k(\gamma) \approx 4^{1/\tau(\gamma/2,\gamma)},
\]
where $\tau(\gamma/2,\gamma)$ is a different constant and can be bound from equation $(7.54)$ as $\tau(\gamma/2,\gamma)\ll \gamma^2$. Together this implies that for a given $m$, we would need $\gamma \gg (\log \log \log m)^{-1/2}$ in order for the result to be non-trivial---that is, we must have at least
\[
m^{C/\sqrt{\log \log \log m}} \le N
\]
for some constant $C$. 

\begin{cor}[Corollary 8.31 in \cite{Bourgain1}]\label{cor:Bourgain1}
Let $\gamma_0  \in (0,1)$. Then there exists $\epsilon_0=\epsilon_0(\gamma_0)>0$ and $\tau_0=\tau_0(\gamma_0)>0$ such that the following holds. Let $m\in \mathbb{Z}_+$ be  a sufficiently large number, let $a\in \mathbb{Z}_m^*$, and assume that there exists $m_1|m$ such that
\begin{enumerate}
\item $m_1>m^{\gamma_0}$; and,
\item if $m'|m_1$ and $m'>m_1^{\epsilon_0}$, then $\ord(b,m') > (m')^{\gamma_0}$.
\end{enumerate}
Then if $N> m^{\gamma_0}$, we have that
\[
\max_{a\in \mathbb{Z}_m^*} \left| \sum_{n=1}^N e\left( \frac{a}{m} b^n\right) \right| < m^{-\tau_0} N.
\]
\end{cor}

With this second result, it is non-trivial over a potentially larger range, but the strength of the bound is the real limiting factor. By using equations $(8.40)$--$(8.42)$, $(8.57)$,  $(8.69)$,   and $(8.72)$ in \cite{Bourgain1}, we see that $\tau_0$ has an upper bound of $O(\epsilon(\gamma_0) \gamma_0^3)$, with $\epsilon$ defined as it was in Theorem \ref{thm:Bourgain1}. In particular $\tau_0$ is bounded by something roughly of the form $1/\exp \exp (1/\gamma_0^2)$. Thus, in order for the bound given by Corollary \ref{cor:Bourgain1} to be $o(N)$, we must have that $\gamma_0 \gg (\log \log \log m)^{-1/2}$, i.e., we would already need to be in the range covered by Theorem \ref{thm:Bourgain1}. Assuming further that $N\le \ord(b,m)$, in order for the bound given by Corollary \ref{cor:Bourgain1} to be better than $N-1$, we must have that $\gamma_0 \gg (	 \log \log m)^{-1/2}$.

\section{The differencing lemma}\label{sec:differencinglemma}

The following lemma we call the differencing lemma because of its similarities to Weyl differencing.

\begin{lem}\label{lem:primarylemma}
Let $b\ge 2$ be a positive integer. Let $a,m,m'$ be integers, $m,m'$ non-zero with $\gcd(m,b)=\gcd(m',b)=1$. Let $\tau = \ord(b,m')$. Then
\[
\left| \sum_{n=1}^N e\left( \frac{a}{m} b^n\right)\right|^2 \le m' N +  2m' \sum_{1 \le i < N/\tau} \left| \sum_{n=1}^{N-i\tau} e\left( \frac{a(b^{i\tau}-1)}{m} b^n\right) \right|.
\]
\end{lem}

In using the differencing lemma, the idea will be to choose $m'$ in such a way so that $a(b^{i\tau}-1)/m$ has a denominator (in lowest terms) much lower than $m$ itself. This is because various estimates, such as Lemma \ref{lem:shortKorobovsum}, are generally strengthened by having smaller denominators.

\begin{proof}
We begin with the following simple relations:
\begin{align*}
\left| \sum_{n=1}^N e\left( \frac{a}{m} b^n\right)\right|^2 &\le \sum_{i=1}^{m'} \left| \sum_{n=1}^N e\left( \left(\frac{a}{m}+\frac{i}{m'}\right) b^n\right)\right|^2 = \sum_{i=1}^{m'} \sum_{n,n'=1}^N e\left( \left(\frac{a}{m}+\frac{i}{m'}\right) (b^{n'}-b^n)\right)\\
&= \sum_{n,n'=1}^N e\left( \frac{a}{m}(b^{n'}-b^n)\right) \sum_{i=1}^{m'} e\left( \frac{i}{m'} (b^{n'}-b^n)\right)\\
&= m' \sum_{n,n'=1}^N e\left( \frac{a}{m}(b^{n'}-b^n) \right) \delta_{m'}(b^{n'}-b^n),
\end{align*}
where $\delta_{m'}(x)$ equals $1$ if $m'|x$ and $0$ otherwise. Note that the terms in the final sum where $n=n'$ contribute to the $m' N$ term from the lemma. For the remaining terms we note that swapping the values of $n$ and $n'$ will take conjugates, and the sum of a number and its conjugate will be at most twice its norm.
Thus,
\begin{equation}\label{eq:keylemma2}
\left| \sum_{n=1}^N e\left( \frac{a}{m} b^n\right)\right|^2 \le m'N + 2m' \left| \sum_{1\le n < n' \le N}  e\left( \frac{a}{m}(b^{n'}-b^n) \right) \delta_{m'}(b^{n'}-b^n)\right|.
\end{equation}

To estimate the right-hand side of \eqref{eq:keylemma2}, note that for this internal sum, $n<n'$ and $\gcd(m',b)=1$. Thus, $\delta_{m'}(b^{n'}-b^n)=\delta_{m'}(b^{{n'}-n}-1)$. But $m'|b^{n'-n}-1$ if and only if $\tau | n'-n$. Let us say that $n'=n+i\tau$. Since the only such pairs of integers $n,n'$ that make the internal sum non-zero are those where $n'=n+i\tau$, we may rewrite the inequality as follows
\begin{align*}
\left| \sum_{n=1}^N e\left( \frac{a}{m} b^n\right)\right|^2 &\le m'N + 2m' \left| \sum_{1\le i < N/\tau} \sum_{n=1}^{N-i\tau}  e\left( \frac{a}{m}(b^{n+i\tau}-b^n) \right) \right|\\
&\le m'N + 2m' \sum_{1\le i < N/\tau} \left| \sum_{n=1}^{N-i\tau} e\left( \frac{a(b^{i\tau}-1)}{m} b^n\right) \right|.
\end{align*}
This completes the proof.
\end{proof}

We could strengthen these results considerably if we make an additional assumption. If we had chosen $a$ relatively prime to $m$ in such a way that we minimized the size of $\left| \sum_{n=1}^N e(ab^n/m)\right|$ and if we assumed that $m'|m$, then we could have removed the copy of $m'$ on the right-hand side of the estimate in Lemma \ref{lem:primarylemma}. This would happen because we would have the following in the first line of the modified proof:
\[
\left| \sum_{n=1}^N e\left( \frac{a}{m} b^n\right)\right|^2 \le \frac{1}{m'}\sum_{i=1}^{m'} \left| \sum_{n=1}^N e\left( \left(\frac{a}{m}+\frac{i}{m'}\right) b^n\right)\right|^2.
\]
In light of how we want to apply Lemma \ref{lem:primarylemma} iteratively in the sequel, this variant does not help us much.

\section{The iterative step}\label{sec:iteration}

In this section, we will prove the following proposition, which makes use of the differencing lemma to turn one estimate into another estimate.

\begin{prop}\label{prop:iterativestep}
Let $\alpha,\gamma,\delta\in [0,1]$, $A,B,r\ge 0$, and $\nu\in [1,2]$, with $1+\gamma\ge \nu$ and $\alpha \neq 0$.

Let $P=\{p_1,p_2,\dots,p_s\}$ be a finite set of primes and suppose that for all $m\in \mathbb{N}_P$, $a\in \mathbb{Z}$ with $\gcd(a,m) =1$, and  $N\in \mathbb{N}$, the following inequality holds: 
\begin{equation}\label{eq:iterativeinitial}
\left| \sum_{n=1}^N e\left( \frac{a}{m } b^n\right)\right| \le \left( A m^{\alpha } N^{\gamma} +B m^{-\delta} N^{\nu}\right)  (1+\log m)^r.
\end{equation}

 Then for all $m\in \mathbb{N}_P$, $a\in \mathbb{Z}$ with $\gcd(a,m) =1$, and   $N\in \mathbb{N}$, the following inequality also holds: 
\begin{equation}\label{eq:iterativeresult}
\left| \sum_{n=1}^N e\left( \frac{a}{m } b^n\right)\right| \le \left( A' m^{\alpha'} N^{\gamma'}+ B' m^{-\delta' } N^{\nu'}\right)\left( 1+ \log m\right)^{r/2},
\end{equation}
where
\begin{align*} 
\alpha' & = \frac{\alpha}{2(1+\alpha)} &  \gamma' &= \frac{1+\gamma+\alpha\nu}{2(1+\alpha)}\\
\delta' &= \frac{\delta}{2(1+\alpha)} & \nu' &= \frac{1+\nu}{2} + \frac{(1+\gamma-\nu)\delta}{2(1+\alpha)}
\end{align*}
\[
A' =\left( 2^{s+2}Q(BC_{P,\delta}+AC_{P,\alpha})+2Q+2AMC_{P,1+\alpha})\right)^{1/2}
\]
\[ 
B' =\max\left\{\left(2^{1+\delta}BMQ^{\delta}C_{P,1-\delta}\right)^{1/2}, 1\right\}
\]
and $Q:=p_1p_2\dots p_s$, $M=M(P)$ as in \eqref{eq:Mdef}, and $C_{P,\cdot}$ is defined in Lemma \ref{lem:phisumbound}.
\end{prop}

We need some lemmas to assist in the proof of this result.

\begin{lem}\label{lem:phisumbound}
Let $P=\{p_1,p_2,\dots,p_s\}$ be a finite set of primes. If $\alpha>0$, then
\[
 \sum_{d|n} d^{\alpha} \le  C_{P,\alpha} \cdot n^{\alpha} \text{ and }  \sum_{d|n} d^{-\alpha} \le  C_{P,\alpha}\]
where 
 \[
 C_{P,\alpha}:= \prod_{i=1}^s \frac{p_i^{\alpha}}{p_i^{\alpha}-1}.
\]
\end{lem}

\begin{proof}
We have that
\begin{align*}
 \sum_{d|n} d^{\alpha} &= \prod_{i=1}^s \left( \sum_{j=0}^{\ell_i} p_i^{\alpha j} \right) = \prod_{i=1}^s \frac{p_i^{\alpha(\ell_i+1)}-1}{p_i^{\alpha}-1}\le  \prod_{i=1}^s \frac{p_i^{\alpha(\ell_i+1)}}{p_i^{\alpha}-1} = n^{\alpha} \prod_{i=1}^s \frac{p_i^{\alpha}}{p_i^{\alpha}-1}.
\end{align*}
On the other hand, we also have that
\begin{align*}
 \sum_{d|n} d^{-\alpha} &= \prod_{i=1}^s \left( \sum_{j=0}^{\ell_i} p_i^{-\alpha j} \right) = \prod_{i=1}^s \frac{1-p_i^{-\alpha(\ell_i+1)}}{1-p_i^{-\alpha}} \le \prod_{i=1}^s \frac{1}{1-p_i^{-\alpha}}= \prod_{i=1}^s \frac{p_i^{\alpha}}{p_i^{\alpha}-1} .
\end{align*}
\end{proof}

\begin{lem}\label{lem:partialphi}
For $d|n$, let $\phi_d(n,x)$ denote the number of integers $i\in [1,x)$ such that $\gcd(i,n)=d$. If $n$ has at most $s$ distinct prime factors, then
\[
\phi_d(n,x) \le \frac{x}{n} \phi(n/d) + 2^s.
\]
Here $\phi(\cdot)$ is the usual Euler totient function.
\end{lem}

\begin{proof}
By Inclusion-Exclusion we have that
\[
\phi_d(n,x) \le \sum_{\ell |n/d} \mu(\ell) \left\lfloor \frac{x}{d\ell} \right\rfloor,
\]
where the fact that this is $\le$ and not $=$ accounts for the possibility that $x$ is an integer and that $\gcd(i,x)=d$.

We note that $\phi(n/d)=n/d \sum_{\ell |n/d} \mu(\ell)/\ell$ and that $z-1\le \lfloor z\rfloor \le z$ for any $z\in \mathbb{R}$. Therefore,
\[
\phi_d(n,x) \le \frac{x}{n} \phi(n/d) + \sum_{\ell |n/d} |\mu(\ell)|,
\]
but the  sum in this equation is just a count of the number of squarefree divisors of $n/d$, which is bounded by $2^s$.
\end{proof}

\begin{proof}[Proof of Proposition \ref{prop:iterativestep}]
Let $m$, $a$, and $N$ be given, satisfying the conditions of the second part of the statement of the proposition. We may assume that $N^{1+\gamma-\nu} <m$. Otherwise, if $N^{1+\gamma-\nu}\ge m$, then the $m^{-\delta'}N^{\nu'}$ term of \eqref{eq:iterativeresult} would be greater than
\[
m^{-\frac{\delta}{2(1+\alpha)}} N^{\frac{1+\nu}{2}+ \frac{(1+\gamma-\nu)\delta}{2(1+\alpha)}}\ge N^{\frac{1+\nu}{2}},
\]
but since $\nu\ge 1$ this is already worse than the trivial bound and, as such, is trivially true. (It was for this estimation that we included the $\max\{\cdot,1\}$ in the definition of $B'$.) We may likewise assume that $m>1$, since otherwise the second term of \eqref{eq:iterativeresult} again is at least the size of the trivial bound and is trivially true.

We will want to apply Lemma \ref{lem:primarylemma} for some choice of $m'|m$ to be given momentarily. If we apply Lemma \ref{lem:primarylemma} as we have said, we obtain
\begin{equation}\label{eq:iterationfirststep}
\left| \sum_{n=1}^N e\left( \frac{a}{m} b^n\right)\right|^2 \le m' N + 2m' \sum_{1\le i < N/\tau} \left| \sum_{n=1}^{N-i\tau} e\left( \frac{a(b^{i\tau} -1 )}{m} b^n\right)\right|,
\end{equation}
where $\tau=\ord(b,m')$.

It will help to optimize later on, if $m'$ is taken to be $m'= m^{\alpha/(1+\alpha)}N^{(1+\gamma-\nu)/(1+\alpha)}$. However, this choice of $m'$ may not even be an integer, let alone have the desired properties to allow us to estimate the sums on the right hand side of \eqref{eq:iterationfirststep}. So instead, let us suppose that $m= p_1^{\ell_1}p_2^{\ell_2 }\dots p_s^{\ell_s}$. Suppose further that $x$ is defined by
\[
m^x = m^{\alpha/(1+\alpha)}N^{(1+\gamma-\nu)/(1+\alpha)}
\]
By our assumptions that $m>1$, $\alpha \neq 0$, and $N^{1+\gamma-\nu}<m$, we have that $0<x<1$.

We will then let 
\begin{equation}\label{eq:m'def}
m' = \begin{cases}
Q p_1^{\lfloor x\ell_1\rfloor} p_2^{\lfloor  x\ell_2\rfloor} \dots p_s^{\lfloor x\ell_s\rfloor}, & \text{if }4\nmid m,\\
Q p_1^{\max\{\lfloor x\ell_1\rfloor,1\}} p_2^{\lfloor  x\ell_2\rfloor} \dots p_s^{\lfloor x\ell_s\rfloor}, & \text{if }4|m\text{ and } p_1=2,
\end{cases}
\end{equation}
where $Q=p_1p_2\dots p_s$. 
This construction guarantees that $m'|m$ (since $x<1$), that every prime dividing $m$ divides $m'$, and that if $4|m$ then $4|m'$ as well.
We also have that 
\begin{equation}\label{eq:m'bound}
m^{\alpha/(1+\alpha)}N^{(1+\gamma-\nu)/(1+\alpha)} \le m' \le 2Q m^{\alpha/(1+\alpha)}N^{(1+\gamma-\nu)/(1+\alpha)}.
\end{equation}

We will briefly show that we may assume that $m' \le 2Q N$ throughout the remainder of the proof. Suppose otherwise, then $N < m'/2Q \le m^{\alpha/(1+\alpha)} N^{(1+\gamma-\nu)/(1+\alpha)}$, which in turn implies that $m^\alpha > N^{\alpha+\nu-\gamma}$. However, if this were true, then the $m^\alpha N^\gamma$ term of \eqref{eq:iterativeresult} would be greater than 
\[
m^{\alpha/2(1+\alpha)} N^{(1+\gamma+\alpha\nu)/2(1+\alpha)} > N^{(\alpha+\nu-\gamma)/2(1+\alpha)} N^{(1+\gamma+\alpha\nu)/2(1+\alpha)} = N^{(1+\nu)/2}.
\] 
(Note: we implicitly use the fact that the $2Q$ term in the definition of $A'$ forces $A'\ge 1$.)
Since $\nu\ge 1$, this implies that the first term of \eqref{eq:iterativeresult} is larger than the trivial bound of $N$, and thus is trivially true.

To begin estimating \eqref{eq:iterationfirststep}, let us consider what the fraction $a(b^{i\tau}-1)/m$ would be in lowest terms. Let $\overline{m} = \gcd(b^{\tau}-1,m)$ and note that $m' | \overline{m}| m$.  Then, for all $i\in \mathbb{N}$ the denominator must be a divisor of $m/\overline{m}$. However, it could be considerably smaller than this. We will make use of a few claims to isolate the behavior of $a(b^{i\tau}-1)/m$.

\begin{claim}\label{claim:1}
We have that $\ord(b,\overline{m})=\tau$. In fact, $\overline{m}$ is the largest divisor of $m$ for which $b$ has multiplicative order $\tau$.
\end{claim}

\begin{proof}[Proof of claim]
Since $\overline{m} | b^\tau-1$, we have that $\ord(b,\overline{m})| \tau$. However, if we had that $\ord(b,\overline{m})=t< \tau$, then since $m' | \overline{m}|b^t-1$, we would have that $\tau=\ord(b,m') | t$, which is a contradiction. 

Suppose there were a divisor $D$ of $m$ larger than $\overline{m}$ for which $b$ had multiplicative order $\tau$. Then $D|b^\tau-1$ and $D|m$, implying that $\overline{m}$ is not the gcd of $b^{\tau}-1$ and $m$, a contradiction.
\end{proof}

\begin{claim}\label{claim:2}
 We have $\gcd(b^{i\tau}-1,m)=\overline{m}d$ (so that the denominator is $m/\overline{m}d$) if and only if $\gcd(i,m/\overline{m})=d$.
 \end{claim}

\begin{proof}[Proof of claim]
To show this, it suffices to prove that for any $d|m/\overline{m}$, we have that $d|i$ if and only if $\overline{m}d|b^{i\tau}-1$. 

We see that $\overline{m}d|b^{i\tau}-1$ if and only if $\ord(b,\overline{m}d) | i\tau$. By Claim \ref{claim:1}, we have that $\ord(b,\overline{m})=\tau$. Putting the divisibility condition $\ord(b,\overline{m}d)|i\tau$ into the language of Lemma \ref{lem:KorobovMultOrder}, we have
\begin{equation}\label{eq:divisibility}
\left. \frac{\overline{m}d}{m_1(\overline{m}d)} \tau'(\overline{m}d) \middle| i \frac{\overline{m}}{m_1(\overline{m})} \tau'(\overline{m}).\right.
\end{equation}
The desired result (namely, that \eqref{eq:divisibility} is equivalent to $d|i$) will follow if we have that  $\tau'(\overline{m}d)=\tau'(\overline{m})$ and $m_1(\overline{m}d)=m_1(\overline{m})$.

According to Lemma \ref{lem:KorobovMultOrder}, there is only one way for $\tau'(\overline{m}d)\neq \tau'(\overline{m})$. Since $\overline{m}d$ and $\overline{m}$ share the same set of prime factors, we have $\tau_1(\overline{m})=\tau_1(\overline{m}d)$, and thus $\tau'(\overline{m})$ and $\tau'(\overline{m}d)$ can differ only if $2|| \overline{m}$, but $4| \overline{m}d$. However, by the construction of $m'$ in \eqref{eq:m'def}, this cannot happen. So we must have $\tau'(\overline{m}d)= \tau'(\overline{m})$.

It remains to consider $m_1(\overline{m})$ and $m_1(\overline{m}d)$. According to Lemma \ref{lem:KorobovMultOrder} again, since $\overline{m}$ and $\overline{m}d$ share the same set of prime factors, the $\beta_i$'s corresponding to $\overline{m}$ are equal to the $\beta_i$'s corresponding to $\overline{m}d$. Thus $m_1(\overline{m})$ and $m_1(\overline{m}d)$ can differ if and only if there is some $p_i | d$ such that $ p_i^{\beta_i} \nmid m_1(\overline{m})$. If such a prime existed, it would imply that $m_1(\overline{m}p_i) = p_i m_1(\overline{m})$ and also that $\tau'(\overline{m})=\tau'(\overline{m}p_i)$ (by the same argument as in the previous paragraph), and thus that $\tau = \ord(b,\overline{m})= \ord(b,\overline{m}p_i)$. However, this is impossible, since by Claim \ref{claim:1}, $\overline{m}$ is the largest divisor of $m$ with order $\tau$.  This completes the proof.
\end{proof}

\begin{claim}\label{claim:3}
We have $(\overline{m}/m') | M$, so that $\overline{m}\le Mm'$
\end{claim}

\begin{proof}[Proof of claim]
By Claim \ref{claim:1}, $\ord(b,m')=\ord(b,\overline{m})$. Thus, in the language of Lemma \ref{lem:KorobovMultOrder}, we have
\[
\frac{m'}{m_1(m')}\tau'(m') = \frac{\overline{m}}{m_1(\overline{m})} \tau'(\overline{m}).
\]
By the same argument as in Claim \ref{claim:2}, we have $\tau'(m')=\tau'(\overline{m})$. Thus $\overline{m}/m' = m_1(\overline{m})/m_1(m')$. The result then follows by noting that both sides of this equality are integers and that $m_1(\overline{m})|M$.
\end{proof}

So in \eqref{eq:iterationfirststep}, we wish to apply the estimates given by \eqref{eq:iterativeinitial}. By the claim above, for each $i$ with $\gcd(i,m/\overline{m})=d$, we have that the denominator of the fraction in lowest terms is $m/\overline{m}d$. Thus, estimating a bit crudely in some places, we obtain the following bounds:
\begin{align*}
&2m' \sum_{1\le i < N/\tau } \left| \sum_{n=1}^{N-i\tau} e\left( \frac{a(b^{i\tau} -1 )}{m} b^n\right)\right| \\
&\qquad = 2m' \sum_{d| m/\overline{m}} \sum_{\substack{1\le i < N/\tau\\ \gcd(i,m/\overline{m})=d}}  \left| \sum_{n=1}^{N-i\tau} e\left( \frac{a(b^{i\tau} -1 )}{m} b^n\right)\right| \\
&\qquad \le 2m' \sum_{d| m/\overline{m}} \sum_{\substack{1\le i < N/\tau\\ \gcd(i,m/\overline{m})=d}} \left( A \left(\frac{m}{\overline{m}d}\right)^{\alpha } (N-i\tau)^{\gamma} +B \left(\frac{m}{\overline{m}d}\right)^{-\delta} (N-i\tau)^{\nu}\right)  \left(1+\log \frac{m}{\overline{m}d}\right)^r\\
&\qquad \le  2m' \sum_{d| m/\overline{m}} \sum_{\substack{1\le i < N/\tau\\ \gcd(i,m/\overline{m})=d}} \left( A \left(\frac{m}{\overline{m}d}\right)^{\alpha } N^{\gamma} +B \left(\frac{m}{\overline{m}d}\right)^{-\delta} N^{\nu}\right)  (1+\log m)^r\\
&\qquad = 2m' \sum_{d|m/\overline{m}} \phi_d\left(\frac{m}{\overline{m}},\frac{N}{\tau}\right)  \left( A \left(\frac{m}{\overline{m}d}\right)^{\alpha } N^{\gamma} +B \left(\frac{m}{\overline{m}d}\right)^{-\delta} N^{\nu}\right)  (1+\log m)^r.
\end{align*}
We next apply Lemma \ref{lem:partialphi} and the fact that $\tau/\overline{m} = \tau'(\overline{m})/m_1(\overline{m}) \ge 1/M$:
\begin{align*}
&2m' \sum_{1\le i < N/\tau} \left| \sum_{n=1}^{N-i\tau} e\left( \frac{a(b^{i\tau} -1 )}{m} b^n\right)\right| \\
&\qquad \le 2m' \sum_{d|m/\overline{m}}  \left( \frac{N/\tau}{m/\overline{m}}\phi\left(\frac{m}{\overline{m}d}\right) + 2^s\right)  \left( A \left(\frac{m}{\overline{m}d}\right)^{\alpha } N^{\gamma} +B \left(\frac{m}{\overline{m}d}\right)^{-\delta} N^{\nu}\right)  (1+\log m)^r\\
&\qquad \le  \frac{2Mm'}{m} \left( A  N^{1+\gamma} \sum_{d|m/\overline{m}} \left( \frac{m}{\overline{m}d}\right)^{\alpha} \phi\left( \frac{m}{\overline{m}d}\right) + B N^{1+\nu} \sum_{d|m/\overline{m}} \left( \frac{m}{\overline{m}d}\right)^{-\delta} \phi\left( \frac{m}{\overline{m}d}\right)   \right)(1+\log m)^r\\
&\qquad \qquad + 2^{s+1}m'\sum_{d|m/\overline{m}}  \left( A \left(\frac{m}{\overline{m}d}\right)^{\alpha } N^{\gamma} +B \left(\frac{m}{\overline{m}d}\right)^{-\delta} N^{\nu}\right)  (1+\log m)^r
\end{align*}
Now, we use that $\phi(n)\le n$ for all $n\in \mathbb{N}$. Also, in each sum, we can do a change of variables to replace $m/\overline{m}d$ with just $d$. Thus, applying Lemma \ref{lem:phisumbound} and the fact that $m' \le \overline{m}$, we obtain:
\begin{align*}
&2m' \sum_{1\le i < N/\tau} \left| \sum_{n=1}^{N-i\tau} e\left( \frac{a(b^{i\tau} -1 )}{m} b^n\right)\right| \\
&\qquad \le  \frac{2Mm'}{m} \left( A  N^{1+\gamma} \sum_{d|m/\overline{m}}  d^{1+\alpha}  + B N^{1+\nu} \sum_{d|m/\overline{m}} d^{1-\delta}    \right)(1+\log m)^r\\
&\qquad \qquad + 2^{s+1}m'\sum_{d|m/\overline{m}}  \left( A d^{\alpha } N^{\gamma} +B d^{-\delta} N^{\nu}\right)  (1+\log m)^r\\
&\qquad \le  \frac{2Mm'}{m} \left( A C_{P,1+\alpha}  \left(\frac{m}{\overline{m}}\right)^{1+\alpha}  N^{1+\gamma}  + B  C_{P,1-\delta} \left(\frac{m}{\overline{m}}\right)^{1-\delta}N^{1+\nu}   \right)(1+\log m)^r\\
&\qquad \qquad + 2^{s+1}m' \left( A C_{P,\alpha} \left(\frac{m}{\overline{m}}\right)^{\alpha } N^{\gamma} +B C_{P,\delta} N^{\nu}\right)  (1+\log m)^r\\
&\qquad \le 2M \left( A  C_{P,1+\alpha} \left( \frac{m}{m'}\right)^{\alpha} N^{1+\gamma} + B  C_{P,1-\delta}\left( \frac{m}{m'}\right)^{-\delta}N^{1+\nu} \right) (1+\log m)^r\\
&\qquad \qquad + 2^{s+1} m' \left( A C_{P,\alpha}  \left( \frac{m}{m'}\right)^{\alpha} N^\gamma + B C_{P,\delta}  N^\nu\right) (1+\log m)^r
\end{align*}

We combine this estimate with our assumption that $m' \le 2Q N$ (so that $m'N^{\gamma} \le 2Q N^{1+\gamma}$) and place it into \eqref{eq:iterationfirststep} to obtain:
\begin{align*}
\left| \sum_{n=1}^N e\left( \frac{a}{m} b^n\right)\right|^2 &\le (2^{s+1}BC_{P,\delta}+1)m' N^{\nu} (1+\log m)^r \\ &\qquad + (2AMC_{P,1+\alpha}+2^{s+2}AQC_{P,\alpha})\left( \frac{m}{m'}\right)^\alpha N^{1+\gamma}(1+\log m)^r\\
&\qquad + 2BMC_{P,1-\delta} \left( \frac{m}{m'}\right)^{-\delta} N^{1+\nu} (1+\log m)^r.
\end{align*}
The proposition then follows by applying our bounds on $m'$ from \eqref{eq:m'bound} and noting that $\sqrt{x+y} \le \sqrt{x}+\sqrt{y}$ for any non-negative values $x,y$.
\end{proof}

\begin{rem}
We made our choice of $m'$ in the proof of Proposition \ref{prop:iterativestep} in order to balance the size of the terms $m' N^\nu$ and the $(m/m')^\alpha N^{1+\gamma}$ in the final inequality. We could alternately have chosen $m'$ so as to balance the size of the terms $(m/m')^{\alpha} N^{1+\gamma}$ and $(m/m')^{-\delta} N^{1+\nu}$. However, it appears, given our calculations in later sections, that this would not be as efficient a choice to make.

We could also have made a choice of $m'$ to attempt to balance out the size of the constants, such as $2^{s+2} A Q C_{P,\alpha}$; however, again in light of calculations in later sections, this would not produce a strong enough difference for us to be concerned with.
\end{rem}

\section{Asymptotic estimates of recursive formulae}\label{sec:recursive}

One of the main theorem of our paper is the following. In this section, we will focus on trying to prove asymptotic estimates for the various recursively defined terms it contains. As a result of these estimates, we will obtain a proof of Theorem \ref{thm:main}

\begin{thm}\label{thm:recursive}Let $P=\{p_1,p_2,\dots,p_s\}$ be a finite set of primes and suppose that $m\in \mathbb{N}_P$, $a\in \mathbb{Z}$ with $\gcd(a,m) =1$, and $k\in \mathbb{N}_{\ge 0}$. Then for all $N\in \mathbb{N}$, we have that
\[
\left| \sum_{n=1}^N e\left( \frac{a}{m} b^n\right)\right| \le \left( A_k m^{\alpha_k} N^{\gamma_k} + B_k m^{-\alpha_k} N^{\nu_k} \right) (1+\log m)^{2^{-k}}
\]
where
\[
\alpha_k := \frac{1}{2^{k+2}-2}
\]
and $\gamma_k$, $\nu_k$, $A_k$ and $B_k$ are defined by
\[
\gamma_0:= 0, \qquad \nu_0 :=1, \qquad A_0 := 1, \qquad B_0 :=M,
\]
when $k=0$ and recusrively by
\begin{align*}
\gamma_k& =\frac{1+\gamma_{k-1} + \alpha_{k-1} \nu_{k-1}}{2(1+\alpha_{k-1})}\\
\nu_k &= \frac{1+\nu_{k-1}}{2} + \frac{(1+\gamma_{k-1}-\nu_{k-1})\alpha_{k-1}}{2(1+\alpha_{k-1})}\\
A_k &= \left( 2^{s+2}Q(A_{k-1}+B_{k-1})C_{P,\alpha_{k-1}}+2Q+2A_{k-1}MC_{P,1+\alpha_{k-1}}\right)^{1/2}\\
 B_k &=\left( 2^{1+\alpha_{k-1}}B_{k-1}MQ^{\alpha_{k-1}} C_{P,1-\alpha_{k-1}}\right)^{1/2},
\end{align*}
when $k\ge 1$. Here we also have $Q:=p_1p_2\dots p_s$, $M=M(P)$ as in \eqref{eq:Mdef}, and $C_{P,\cdot}$ is defined in Lemma \ref{lem:phisumbound}.
\end{thm}

We note that the $k=0$ case is given exactly by Lemma \ref{lem:longKorobovsum}. We then prove Theorem \ref{thm:recursive} via induction, using Proposition \ref{prop:iterativestep} as the induction step. The only thing that would need to be proven is that $1+\gamma_k \ge \nu_k$, $\gamma_k\in [0,1]$, $\nu_k\in [1,2]$ for all $k$ and that $\alpha_k$ has the desired form.

For $\alpha_k$, this follows by a simple induction argument which we will not show here.

For $\gamma_k$ and $\nu_k$, we note that for $k\ge 1$, we have
\begin{align*}
\gamma_k + \nu_k &=\frac{1+\gamma_{k-1} + \alpha_{k-1} \nu_{k-1}}{2(1+\alpha_{k-1})}+  \frac{1+\nu_{k-1}}{2} + \frac{(1+\gamma_{k-1}-\nu_{k-1})\alpha_{k-1}}{2(1+\alpha_{k-1})}\\
&= 1+\frac{\gamma_{k-1}+\nu_{k-1}}{2}.
\end{align*}
It can be shown inductively that
\begin{equation}\label{eq:gammanusum}
\gamma_k + \nu_k = 2-\frac{1}{2^k}.
\end{equation}

Moreover, we can consider the difference, which simplies to
\begin{align*}
\nu_k-\gamma_k &=   \frac{1+\nu_{k-1}}{2} + \frac{(1+\gamma_{k-1}-\nu_{k-1})\alpha_{k-1}}{2(1+\alpha_{k-1})} - \frac{1+\gamma_{k-1} + \alpha_{k-1} \nu_{k-1}}{2(1+\alpha_{k-1})} \\
&= \frac{\alpha_{k-1}}{1+\alpha_{k-1}} + \frac{1-\alpha_{k-1}}{2(1+\alpha_{k-1})}(\nu_{k-1}-\gamma_{k-1})\\
&= \frac{1}{2^{k+1}-1} + \left( \frac{1}{2} - \frac{1}{2^{k+1}-1}\right) (\nu_{k-1}-\gamma_{k-1})
\end{align*}

We will note that it follows almost immediately from these equations that $1+\gamma_k \ge \nu_k$ for all $k$. In particular, we know it holds for $k=0$ and then inductively we have that for all $k\ge 1$
\[
\nu_k-\gamma_k \le \frac{1}{2^{k+1}-1} + \frac{1}{2}\cdot (\nu_{k-1}-\gamma_{k-1})\le \frac{1}{3} + \frac{1}{2} < 1.
\]
These formulas also tell us that $\gamma_k$ is always in $[0,1]$ and $\nu_k$ is always in $[1,2]$. In particular the formula for $\nu_k$ itself implies that $\nu_k \ge 1$, and since $\nu_k+\gamma_k <2$, we must have that $\gamma_k <1$. 

Together, this is enough to prove Theorem \ref{thm:recursive}. However, to prove Theorem \ref{thm:main}, we need some stronger asymptotics on the various quantities in question, which will be given below in \eqref{eq:gammanuexact}, Lemma \ref{lem:CPbound}, and Lemma \ref{lem:ABbound}.

The difference $\nu_k-\gamma_k$ could be written in a somewhat explicit form in the following way,
\[
\nu_k-\gamma_k = \sum_{j=0}^k \left( \frac{1}{2^{j+1}-1} \cdot \prod_{i=j}^{k-1} \left( \frac{1}{2} - \frac{1}{2^{i+2}-1} \right)\right),
\]
where if $j=k$ then the final product is just $1$. This solution can be proved by a simple induction argument again. We note that this formula may be rewritten as
\[
\nu_k-\gamma_k = \frac{1}{2^{k+1}} \sum_{j=0}^k\left(  \left( 1+ \frac{1}{2^{j+1}-1}\right) \prod_{i=j}^{k-1} \left( 1-\frac{2}{2^{i+2}-1}\right)    \right).
\]  
It is clear that the summands converge to $1$ as $j$ and $k$ tend to infinity, so that $\nu_k-\gamma_k \approx (k+1)/2^{k+1}$, but we need a slightly sharper asymptotic, which we give in the lemma below.

\begin{lem}
There exists a constant $c$ such that
\begin{equation}\label{eq:gammanudiff}
\nu_k-\gamma_k = \frac{k+1+c}{2^{k+1}}+O\left(\frac{k+7}{4^k}\right),
\end{equation}
with implicit constant $1$.
\end{lem}

We state without proof that
\begin{equation}\label{eq:cexplicit}
c= \sum_{j=0}^\infty \left( \left( 1+ \frac{1}{2^{j+1}-1}\right) \prod_{i=j}^{\infty} \left( 1-\frac{2}{2^{i+2}-1}\right)  -1  \right)\approx -1.17094960687104654952
\end{equation}

\begin{proof}
Let $\epsilon_k$ be given by
\[
\epsilon_k := \nu_k-\gamma_k - \frac{k+1}{2^{k+1}},
\]
for $k\ge 0$, with $\epsilon_0= 1/2$.

Then, using our earlier recursion relation on $\nu_k-\gamma_k$, we have that
\begin{align*}
\nu_k-\gamma_k &= \frac{1}{2^{k+1}-1} + \left( \frac{1}{2} - \frac{1}{2^{k+1}-1}\right) (\nu_{k-1}-\gamma_{k-1})\\
&= \frac{1}{2^{k+1}} +\frac{1}{2^{k+1}(2^{k+1}-1)} +\left( \frac{1}{2} - \frac{1}{2^{k+1}-1}\right)\left( \frac{k}{2^{k}} +\epsilon_{k-1}\right)\\
&= \frac{k+1}{2^{k+1}} +\frac{-k+1/2}{2^k(2^{k+1}-1)} +\left( \frac{1}{2} - \frac{1}{2^{k+1}-1}\right) \epsilon_{k-1},
\end{align*}
for all $k\ge 1$. This implies that 
\[
\epsilon_k = \left( \frac{1}{2} - \frac{1}{2^{k+1}-1}\right) \epsilon_{k-1}+\frac{-k+1/2}{2^k(2^{k+1}-1)}, \qquad k\ge 1.
\]
Now let $\epsilon'_k = \epsilon_k \cdot 2^{k+1}$, so that we now have
\begin{equation}\label{eq:epsilonkrelation}
\epsilon'_k = \left( 1 -  \frac{2}{2^{k+1}-1}\right) \epsilon'_{k-1} +\frac{-2k+1}{2^{k+1}-1}, \qquad k\ge 1.
\end{equation}
Taking absolute values and estimating crudely, we see that
\[
|\epsilon'_k| \le |\epsilon'_{k-1}| + \frac{2k-1}{2^{k+1}-1}, \qquad k\ge 1.
\]
However, applying this inequality inductively, we see that $|\epsilon'_k|$ is bounded by
\begin{equation}\label{eq:eprimekbound}
|\epsilon'_k| \le |\epsilon'_0| + \sum_{j=1}^{k} \frac{2j-1}{2^{j+1}-1} \le 1+ \sum_{j=1}^k \frac{j}{2^{j-1}}\le 1+ \sum_{j=1}^\infty \frac{j}{2^{j-1}}=5,
\end{equation}
for all $k\ge 0$.

Returning to  \eqref{eq:epsilonkrelation} and using \eqref{eq:eprimekbound}, we have that
\[
\left| \epsilon'_k -\epsilon'_{k-1}\right| = \left| \frac{-2\epsilon'_{k-1}}{2^{k+1}-1} +\frac{-2k+1}{2^{k+1}-1}\right|  \le \frac{2k+9}{2^{k+1}-1}\le \frac{k+5}{2^{k-1}}, \qquad k \ge 1.
\]
This implies that the $\epsilon'_k$ form a Cauchy sequence and thus converge to some constant $c$. Moreover, it implies that the error $|\epsilon'_k-c|$ must be bounded by 
\[
|\epsilon'_k-c|\le \sum_{j=k+1}^\infty \left| \epsilon'_{j}-\epsilon'_{j-1}\right|\le \sum_{j=k+1}^\infty \frac{j+5}{2^{j-1}}= \frac{k+7}{2^{k-1}}.
\]
We therefore have
\[
\epsilon_k = \frac{c}{2^{k+1}} + O\left( \frac{k+7}{4^k}\right), \qquad k\ge 0,
\]
with implicit constant $1$.
This completes the proof.
\end{proof}

From \eqref{eq:gammanusum} and \eqref{eq:gammanudiff} we see that 
\begin{equation}\label{eq:gammanuexact}
\gamma_k = 1-\frac{k+c+3}{2^{k+2}}+O\left(\frac{k+7}{2^{2k+1}}\right) \qquad \nu_k =1+\frac{k+c-1}{2^{k+2}}+O\left(\frac{k+7}{2^{2k+1}}\right),
\end{equation}
with implicit constant equal to $1$.

\begin{lem}\label{lem:CPbound}
Let $P=\{p_1,p_2,\dots,p_s\}$ be a fixed, finite set of primes.

Then the quantities $C_{P,1-\alpha_{k}}$ and $C_{P,1+\alpha_k}$ are uniformly bounded by $C_{P,1/2}$ and $C_{P,1}$ respectively for all $k\ge 0$. We also have that
\[
C_{P,\alpha_k} \le \prod_{i=1}^s  \frac{2^k}{1-p_i^{-1}}, \qquad k \ge 0.
\]
\end{lem}

\begin{proof}
We note that for any prime $p$ that $p^x/(p^x-1)$ is decreasing for $x>0$. Therefore, since $0\le \alpha_k\le 1/2$ we have that
\[
C_{P,1-\alpha_k} \le C_{P,1/2} = \prod_{i=1}^s \frac{p_i^{1/2}}{p_i^{1/2}-1}, \qquad k\ge 0,
\]
and
\[
C_{P,1+\alpha_k} \le C_{P,1} = \prod_{i=1}^s \frac{p_i}{p_i-1}, \qquad k\ge 0,
\]

Again, since $p^x/(p^x-1)$ is a decreasing function for $x>0$, we may bound $C_{P,\alpha_k}$ by $C_{P,1/2^{k+1}}$. Moreover, for positive $x$, we may consider the following factorization of $1-x$:
\[
1-x = (1-x^{1/2})(1+x^{1/2})= (1-x^{1/4})(1+x^{1/4})(1+x^{1/2}) = \dots = (1-x^{1/2^{k+1}})\prod_{j=1}^{k+1} (1+x^{1/2^j})
\]
For any $0<x<1$, we have that $1+x^{1/2^j}\le 2$, so that the product here is bounded by $2^k$.

Therefore, we have that
\[
C_{P,\alpha_k} \le C_{P,1/2^{k+1}}= \prod_{i=1}^s \frac{1}{1-p^{-1/2^{k+1}}} \le \prod_{i=1}^s \frac{2^k}{1-p_i^{-1}}
\]
as desired.
\end{proof}

\begin{lem}\label{lem:ABbound}
Let $P=\{p_1,p_2,\dots,p_s\}$ be a fixed, finite set of primes. 
We have that \[
A_k \le 6\cdot 2^{s(k+1)+7/2}Q^{3/2} M^2 C_{P,1/2}  \prod_{i=1}^s \frac{1}{1-p_i^{-1}}, \qquad k \ge 0,
\]
 and 
\[
B_k \le 2^{3/2} M Q^{1/2} C_{P,1/2}, \qquad k \ge 0.
\]
\end{lem}

\begin{proof}
By Lemma \ref{lem:CPbound}, we have that $C_{P,1-\alpha_k} \le C_{P,1/2}$. Moreover, we know that $\alpha_k \le 1/2$. Therefore,
\[
B_k \le 2^{3/4} M^{1/2} Q^{1/4} C_{P,1/2}^{1/2} B_{k-1}^{1/2}, \qquad k\ge 1.
\]
Iterating this relation gives
\[
B_k \le 2^{3(1-2^{-k})/2}M^{1-2^{-k}}Q^{(1-2^{-k})/2}C_{P,1/2}^{1-2^{-k}} B_0^{2^{-k}}, \qquad k\ge 1.
\]
Noting that $B_0=M$, we obtain the desired inequality from here.

We recall that
\[
A_k = \left( 2^{s+2}Q(A_{k-1}+B_{k-1})C_{P,\alpha_{k-1}}+2Q+2A_{k-1}MC_{P,1+\alpha_{k-1}}\right)^{1/2}.
\]
We note that $x+y+z\le 3xyz$ for any $x,y,z\ge 1$ and that $C_{P,1+\alpha_{k-1}}\le C_{P,\alpha_{k-1}}$ by the argument of Lemma \ref{lem:CPbound}. Therefore, by applying these facts and also our bound on $B_k$ provided earlier in this proof, we have
\begin{align*}
A_k &\le \left( 2^{s+2}Q(A_{k-1}+B_{k-1})C_{P,\alpha_{k-1}}+2Q+2A_{k-1}MC_{P,\alpha_{k-1}}\right)^{1/2}\\
&\le \left(2^{s+2}QM(A_{k-1}+B_{k-1})C_{P,\alpha_{k-1}} + 2^{s+2}QMC_{P,\alpha_{k-1}}+ 2^{s+2}QM A_{k-1}C_{P,\alpha_{k-1}}\right)^{1/2}\\
&\le 2^{(s+2)/2} Q^{1/2} M^{1/2}C_{P,\alpha_{k-1}}^{1/2} \left( 2A_{k-1}+B_{k-1}+1\right)^{1/2} \\
&\le 6^{1/2} 2^{(s+2)/2} Q^{1/2} M^{1/2}C_{P,\alpha_{k-1}}^{1/2} B_{k-1}^{1/2} A_{k-1}^{1/2}\\
&\le 6^{1/2} 2^{(s+2)/2 + 3/4} Q^{3/4} M C_{P,1/2}^{1/2} C_{P,\alpha_{k-1}}^{1/2} A_{k-1}^{1/2}.
\end{align*}
Applying the same iteration to this inequality as we did with $B_k$ and noting that $A_0=1$, we obtain
\[
A_k \le 6\cdot 2^{s+7/2} Q^{3/2} M^2 C_{P,1/2} \prod_{j=0}^{k-1} C_{P,\alpha_j}^{1/2^{k-j}}.
\]
Applying the bound on $C_{P,\alpha_k}$ from Lemma \ref{lem:CPbound}, we obtain
\begin{align*}
A_k &\le 6\cdot 2^{s+7/2} Q^{3/2} M^2 C_{P,1/2} \prod_{j=0}^{k-1} \left( \prod_{i=1}^s 2^{j/2^{k-j}}(1-p_i^{-1})^{-1/2^{k-j}}\right)\\
&\le 6\cdot 2^{s(k+1)+7/2} Q^{3/2} M^2 C_{P,1/2}  \prod_{i=1}^s \frac{1}{1-p_i^{-1}}
\end{align*}
as desired.
\end{proof}

\begin{rem}
The method of bounding $A_k$ was very crude; however, given that $A_k \ge (A_{k-1}C_{P,\alpha_{k-1}})^{1/2}$, it cannot be substantially improved.
\end{rem}

Now we can give the values of $K_1,K_2,K_3$ in Theorem \ref{thm:main} explicitly as follows, combining the work above with \eqref{eq:Mbound}:
\begin{align}
K_1 &= 6\cdot 2^{s+7/2} Q^{3/2} b^{4Q}\prod_{i=1}^s \frac{p_i^{1/2}}{(p_i^{1/2}-1) (1-p_i^{-1})} ,\label{eq:K1def}\\
K_2 &= 2^s ,    \\
K_3 &= 2^{3/2} Q^{1/2} b^{2Q}  \prod_{i=1}^s \frac{p_i^{1/2}}{p_i^{1/2}-1}, \label{eq:K3def}
\end{align}
where $Q= p_1p_2\dots p_s$.

\section{Non-triviality of bounds}\label{sec:effectivity}

Theorems \ref{thm:main} and \ref{thm:recursive} are only useful and interesting if they beat the trivial bound of $N$. In this section we will elaborate more on when they can beat the trivial bound, and thus give proofs of Corollaries \ref{cor:first} and \ref{cor:second}.

\subsection{Non-trivial range and the proof of Corollary \ref{cor:first}}

For a given choice of $k$, let us call the non-trivial range the values of $N$ for which both $m^{\alpha_k} N^{\gamma_k}$ and $m^{-\alpha_k}N^{\nu_k}$ are both less than or equal to $N$. In particular, the non-trivial range is given by
\[
\{m^x: x\in I_k\}, \text{ where }
I_k := \left[ \frac{\alpha_k}{1-\gamma_k}, \frac{\alpha_k}{\nu_k-1}\right].
\]
we have that $\nu_k=1$ for $k=0$ and $k=1$, thus the right endpoint of $I_0$ and $I_1$ is infinity in these cases.

\begin{lem}\label{lem:nontrivialinterval}
We have that $I_k \cap I_{k+1}$ is a non-empty interval for every $k\ge 0$.
\end{lem}

As a consequence of this lemma, for any $N>1$, there exists $k\in \mathbb{N}_{\ge 0}$ such that $N= m^x$ with $x\in I_k$, so that every $N$ falls into the non-trivial range, for some choice of $k$.  

\begin{proof}
Using our earlier asymptotic results \eqref{eq:gammanuexact}, we can show that
\begin{align*}
\frac{\alpha_k}{1-\gamma_k} &= \frac{1}{(k+c+3)\frac{2^{k+2}-2}{2^{k+2}} +O\left( \frac{k+7}{2^{k-1}}\right) }= \frac{1}{k+c+3 +O\left(\frac{ |k+c+3|}{2^{k+1}}\right)+O\left( \frac{k+7}{2^{k-1}}\right) }\\ &=\frac{1}{k+c+3 +O\left( \frac{k+7}{2^{k-2}}\right) },
\end{align*}
and a similar formula for $\alpha_k/(\nu_k-1)$, so that
\begin{equation}\label{eq:Ikasymptotic}
I_k = \left[ \frac{1}{k+c+3+O\left( \frac{k+7}{2^{k-2}}\right)},\frac{1}{k+c-1+O\left( \frac{k+7}{2^{k-2}}\right)}\right], \qquad k \ge 2,
\end{equation}
with implicit constant $1$.
This asymptotic implies that $I_k$ and $I_{k+1}$ overlap on a non-empty interval provided $k\ge 5$. 

We calculate $I_k$ directly for $0\le k \le 5$:
\begin{align}
I_0 &= \left[ \frac{1}{2}, \infty \right) &
I_1 &= \left[ \frac{1}{3}, \infty \right) \notag \\
I_2 &= \left[ \frac{1}{4}, 2\right] &
I_3 &= \left[ \frac{28}{139}, \frac{14}{17}\right] \label{eq:I05}\\
I_4 &= \left[ \frac{105}{622}, \frac{840}{1721}\right] &
I_5 &= \left[ \frac{52080}{358871}, \frac{26040}{76903}\right]. \notag
\end{align}
From these, we can see that $I_k \cap I_{k+1}$ is a non-empty interval for all $k\ge 0$.
\end{proof}

\begin{lem}\label{lem:maxversion}
Let $I$ be a sub-interval of $I_k$ that does not share endpoints with $I_k$. Then there exists a $\delta(I)>0$ such that 
\[
\max\left\{m^{\alpha_k} N^{\gamma_k}, m^{-\alpha_k} N^{\nu_k}\right\}\le m^{-\delta(I)} N.
\]
\end{lem}

\begin{proof}
We note that for all $k$, $\gamma_k \in [0,1)$ and $\nu_k \in [1,2)$. Thus, $m^\alpha N^{\gamma_k-1}$ is a decreasing function in $N$ and $m^{-\alpha} N^{\nu_k-1}$ is a non-decreasing function in $N$. 

Let $I_k = [a,b]$ and $I=[c,d]$. The number $a$ satisfies $m^\alpha (m^a)^{\gamma_k-1}=1$. Thus, since $c>a$, there must exist some $\delta_1>0$ such that $m^\alpha N^{\gamma_k-1} \le m^{-\delta_1}$ for all $N=m^x$ with $x\in I$. A similar argument shows that there exists $\delta_2>0$ such that $m^{-\alpha}N^{\nu_k-1} \le m^{-\delta_2}$ for all $N=m^x$ with $x\in I$. Choosing $\delta(I) = \min\{\delta_1,\delta_2\}$ completes the proof.
\end{proof}

From here it is relatively simple to prove Corollary \ref{cor:first}. First we note that it holds for $N\ge m$ by Lemma \ref{lem:longKorobovsum}, so we may restrict ourself to showing that it holds for $m^\epsilon\le N \le m$. We let $k$ be such that $\epsilon \in I_k$ and $\epsilon$ is not an endpoint of $I_k$. Then, the interval $[\epsilon,1]$ may be written as a disjoitn union of intervals $J_1, J_2, \dots, J_k$ where each $J_i \subset I_i$ and $J_i$ does not share either endpoint with $I_i$. The corollary then holds with $C = 2K_1 K_2^k K_3$ and $\delta = \min_{1\le i \le k} \delta(J_i)$, with $K_1, K_2, K_3$ defined in \eqref{eq:K1def}--\eqref{eq:K3def} and $\delta(J_i)$ defined in Lemma \ref{lem:maxversion}.

\subsection{Optimal range and the proof of Corollary \ref{cor:second}}

We would like to know, for a given $N$, which choice of $k$ is optimal. For example, if $N\in [m^1,m^2]$, then $N$ is in the non-trivial range for $k=0,1,2$. Which of these produces the best bound, ignoring, for the purposes of estimation, the influence of $K_1, K_2, K_3$ and $(1+\log m)^{2^{-k}}$?

We will not give an exact answer to this question, but instead give some quick estimations without proof. First, we can show that $m^{\alpha_k} N^{\gamma_k} = m^{-\alpha_k} N^{\nu_k}$  when $N = m^{y_k}$ with $y_k \approx 1/(k+c+1)$. Thus, when $N= m^x$ with $x\in [\alpha_k/(1-\gamma_k),y_k]$, we expect $m^{\alpha_k} N^{\gamma_k}$ to be the dominant term, while when $x\in [y_k, \alpha_k/(\nu_k-1)]$ we expect $m^{-\alpha_k} N^{\nu_k}$ to be the dominant term. 

Moreover, one can show that $m^{\alpha_k}N^{\gamma_k} = m^{-\alpha_{k+1}}N^{\nu_{k+1}}$ when $N=m^{z_k}$ with $z_k \approx 1/(k+c+2)$. This, combined with the previous estimation, suggests that the point where $k+1$ provides a better estimation than $k$ should be around the point when $N=m^{z_k}$. 

Thus, let us define the optimal range for a given $k$ by
\[
\{m^x : x\in \tilde{I}_k\}, \text{ where }\tilde{I}_k := \left[ \frac{1}{k+c+2}, \frac{1}{k+c+1}\right], \qquad k \ge 1.
\]
By \eqref{eq:Ikasymptotic}, we have that $\tilde{I}_k\subset I_k$ for $k \ge 6$. (Because for $k\ge 6$, we have that $(k+7)/2^{k-2} \le 1$.) By \eqref{eq:I05} and \eqref{eq:cexplicit}, we have that $\tilde{I}_k \subset I_k$ for $1\le k \le 5$. We note that $\tilde{I}_0$ is not well defined since the right-hand endpoint is negative.

\begin{lem}
Let $k\ge 1$ and let $N=m^x $ with $x\in \tilde{I}_k$. Then
\[
\max\left\{m^{\alpha_k} N^{\gamma_k} , m^{-\alpha_k}N^{\nu_k}\right\}\le N^{1-\frac{1}{2^{k+3}}}.
\]
\end{lem}

\begin{proof}
Let $N=m^x$ with $x\in \tilde{I}_k$. We follow the idea of the proof of Lemma \ref{lem:maxversion}. In particular, we have that $m^{\alpha_k} N^{\gamma_k-1}$ should be maximized when $N$ is as small as possible, i.e., when $N= m^{1/(k+c+2)}$, and $m^{-\alpha_k} N^{\nu_k-1}$ should be maximized when $N$ is as large as possible, i.e., when $N=m^{1/(k+c+1)}$.

 We have that 
\[
m^{\alpha_k}N^{\gamma_k-1} \le N^{\alpha_k(k+c+2) + \gamma_k -1} ,
\]
and applying \eqref{eq:gammanuexact}, we get
\begin{align*}
\alpha_k(k+c+2) + \gamma_k -1 &= \frac{k+c+2}{2^{k+2}-2} -\frac{k+c+3}{2^{k+2}} + O\left( \frac{k+7}{2^{2k+1}}\right)\\
&= -\frac{1}{2^{k+2}} + \frac{k+c+2}{2^{k+1}(2^{k+2}-2)} + O\left( \frac{k+7}{2^{2k+1}}\right)\\
&= - \frac{1}{2^{k+2}} + O\left( \frac{k+7}{2^{2k}}\right).
\end{align*}
This will be less than $-1/2^{k+3}$ provided $k\ge 7$. By direct calculation, we have
\begin{align*}
\alpha_1(c+3) + \gamma_1 -1 &\approx -0.195158\\
\alpha_2(c+4) + \gamma_2 -1 &\approx -0.0836393\\
\alpha_3(c+5) + \gamma_3 -1 &\approx -0.0378412\\
\alpha_4(c+6) + \gamma_4 -1 &\approx -0.0176574\\
\alpha_5(c+7) + \gamma_5 -1 &\approx -0.0084263\\
\alpha_6(c+8) + \gamma_6 -1 &\approx -0.0040877
\end{align*}
and thus, $\alpha_k(k+c+2)+\gamma_k-1 \le -1/2^{k+3}$ for $k\ge 1$.

We also have that
\[
m^{-\alpha_k} N^{\nu_k -1} \le N^{-\alpha_k(k+c+1) + \nu_k - 1},
\]
and applying  \eqref{eq:gammanuexact} again, we get
\begin{align*}
-\alpha_k(k+c+1) + \nu_k -1 &= \frac{k+c+1}{2^{k+2}-2} -\frac{k+c-1}{2^{k+2}} + O\left( \frac{k+7}{2^{2k+1}}\right)\\
&= -\frac{1}{2^{k+1}} - \frac{k+c+1}{2^{k+1}(2^{k+2}-2)} + O\left( \frac{k+7}{2^{2k+1}}\right)\\
&= - \frac{1}{2^{k+1}} + O\left( \frac{k+7}{2^{2k}}\right).
\end{align*}
This is less than $-1/2^{k+3}$ provided $k\ge 6$. By direct calculation, we have
\begin{align*}
-\alpha_1(c+1) + \nu_1-1 &\approx -0.138175\\
-\alpha_2(c+2)+\nu_2-1 & \approx -0.0949322\\
-\alpha_3(c+3)+\nu_3-1 & \approx-0.0538255\\
-\alpha_4(c+4)+\nu_4-1 & \approx-0.0287136\\
-\alpha_5(c+5)+\nu_5-1 & \approx-0.0148872
\end{align*}
and thus $-\alpha_k(k+c+1) + \nu_k -1\le -1/2^{k+3}$ for all $k\ge 1$.
\end{proof}

The bulk of Corollary \ref{cor:second} follows almost immediately from this lemma, Theorem \ref{thm:recursive}, and Lemma \ref{lem:ABbound}, using that
\[
 A_k m^{\alpha_k} N^{\gamma_k} + B_k m^{-\alpha} N^{\nu_k}  \le 2A_k B_k \max\left\{m^{\alpha_k} N^{\gamma_k} , m^{-\alpha_k}N^{\nu_k}\right\}.
\]

For the final part of Corollary \ref{cor:second}, we need to know how large $k$ can be and still have
\[
C_1 \cdot C_2^k \cdot N^{1-\frac{1}{2^{k+3}}} (1+\log m)^{\frac{1}{2^{k}}} \le N
\]
for all $N =m^x$ with $x\in \tilde{I}_k$. Taking $2^{k+3}$th powers and rearranging this inequality gives
\begin{equation}\label{eq:finalinequality}
C_1^{2^{k+3}} \cdot C_2^{k2^{k+3}} \le \frac{N}{(1+\log m)^8} \le \frac{m^{\frac{1}{k+c+1}}}{(1+\log m)^8}
\end{equation}

Let $\epsilon>0$ and 
\[
k \le \log_2 \log m - (2+\epsilon)\log_2\log \log m.
\]
We see that once $m$ is sufficiently large, \eqref{eq:finalinequality} will be satisfied for any such $k$, since the left-hand side of \eqref{eq:finalinequality} will be at most on the order of $\exp(C \log m/(\log_2 \log m)^{1+\epsilon})$ for some large $C>0$,  while the right-hand side of \eqref{eq:finalinequality} will be at least on the order of $\exp(C' \log m/\log_2 \log m)$ for some small $C'>0$. This completes the proof of Corollary \ref{cor:second}.

\subsection{Proof of Theorem \ref{thm:secondary}}

Let us suppose, as in the statement of the theorem, that 
\[
\exp\left( \frac{\log m}{\log_2 \log m  -3 \log_2 \log \log m}\right) \le N.
\]
For any fixed $1/2>\epsilon>0$, the result of the theorem holds for $m^{\epsilon} \le N$ by Corollary \ref{cor:first}. For the remaining values not covered by Corollary \ref{cor:first}, we have that there exists some $k$ satisfying $1 \le k \le \log_2 \log m - (5/2)\log_2 \log \log m$ such that $m^{1/(k+c+2)} \le N \le m^{1/(k+c+1)}$, provided $m\in \mathbb{N}_P$ is sufficiently large. By applying Corollary \ref{cor:second}, we have that
\begin{align*}
\frac{1}{N} \left| \sum_{n=1}^N e\left( \frac{a}{m}b^n \right)\right| \le C_1 \cdot C_2^k\cdot  N^{-\frac{1}{2^{k+3}}} (1+\log m)^{2^{-k}}\le C_1 \cdot C_2^{k} \cdot m^{-\frac{1}{(k+c+2)2^{k+3}}} (1+\log m)^{2^{-k}}.
\end{align*}
But since $k\le \log_2 \log m - (5/2) \log_2\log\log m$, it is clear that  the right-hand side of this inequality is no more than
\begin{equation}\label{eq:secondaryfinalresult}
\exp\left( - c (\log \log m)^{3/2}\right)
\end{equation}
for some small $c>0$, as desired.

\begin{rem}
We note that the power of $3/2$ in \eqref{eq:secondaryfinalresult} could be replaced by $2-\epsilon$ for any small $\epsilon>0$.
\end{rem}

\section{Applications}\label{sec:applications}

\subsection{The distribution of digits in rational numbers}

Suppose that $1\le a < m$ with $\gcd(a,m)=1$ and let $b\ge 2$ be relatively prime to $m$ as well. Consider the base-$b$ expansion of $a/m$ given by
\[
\frac{a}{m} = 0.a_1a_2a_3a_4\dots.
\]
Given a finite string of base-$b$ digits $s=[d_1,d_2,\dots,d_k]$, we let
\[
\mathcal{N}_{a,m,s}(N) = \#\{1\le n \le N : a_{n-1+i} = d_i, 1\le i \le k\}
\]
denote the number of appearances of $s$ in the base-$b$ expansion of $a/m$ starting in the first $N$ positions. Results about $\mathcal{N}_{a,m,s}(N)$ are a recurring theme in Korobov's works, with many of them included in his book \cite{KorobovBook}. Here, we are interested in the following result.

\begin{thm}[Theorem 32 in \cite{KorobovBook}]
In addition to the assumptions at the start of this section, assume that $m$ is odd, $m$ has prime factorization $p_1^{\alpha_1}p_2^{\alpha_2} \dots p_s^{\alpha_s}$ with $\alpha_\nu \ge 2$ for $\nu=1,2,\dots,s$. If $\tau = \ord(b,m)$, then for all $N\le \tau$, all finite strings $s$, and all $\epsilon>0$, we have
\[
\mathcal{N}_{a,m,s}(N) = \frac{1}{b^k}N + O\left( m^{\frac{1}{2}+\epsilon}\right),
\] 
where the implied constant only depends on $\epsilon$.
\end{thm}

We note that some of the assumptions that are made in the above theorem are made for convenience. For example, avoiding the case of $m$ even seems to have been done merely to simplify the proof and a similar result should still hold in the even case. 

The above result is non-trivial provided the period $\tau$ is sufficiently large and $N\ge m^{1/2+\epsilon}$. By using the new results of this paper, we extend the range of non-trivial $N$'s considerably in the following theorem, although with a weaker error term in general.

\begin{thm}
Let $P$ be a finite set of primes, $b\ge 2$ be an integer relatively prime to every prime in $P$, and let $0< \epsilon < 1$. Then for all sufficiently large $m\in \mathbb{N}_P$, all $a\in \mathbb{N}$ with $1\le a <m$ and $\gcd(a,m)=1$, all base-$b$ strings $s=[d_1,d_2,\dots,d_k]$, and all $N$ satisfying
\[
\exp\left( \frac{(1+\epsilon)\log m}{\log_2\log m}\right) \le N,
\]
we have that
\[
\mathcal{N}_{a,m,s}(N) = \frac{1}{b^k} N + O\left( N \exp\left( -c(\log \log m)^{3/2}\right)\right),
\]
for some $c>0$ dependent on $P$. The implicit constant may depend on $P$, $b$, and $\epsilon$.

\end{thm}

\begin{proof}
We follow the proof of Theorem 32 of Korobov up to line (325). At this point we have that
\begin{equation}\label{eq:firstRstep}
\mathcal{N}_{a,m,s}(N) =\frac{1}{b^k} N + R,
\end{equation}
where 
\begin{equation}\label{eq:Rbound}
|R| < \sum_{z=1}^{m-1} \frac{1}{z} \left| \sum_{n=0}^{N-1} e\left( \frac{az b^n}{m}\right)\right|.
\end{equation}
It should be noted that the more restrictive conditions Korobov placed on his theorem have not come into play by this point in time.

We fix some $\delta \in (0,1)$ (the actual value does not matter) and break the sum over $z$ in \eqref{eq:Rbound} into the ``big" part where $\gcd(z,m)< m^{\delta}$ and the ``small" part where $\gcd(z,m) \ge m^{\delta}$. (In order to simplify notation, we will drop the $\gcd$ for the remainder of this proof.)

For the small part, we have the following bounds:
\begin{align*}
\sum_{\substack{1\le z \le m-1\\(z,m)\ge m^\delta}} \frac{1}{z} \left| \sum_{n=0}^{N-1} e\left( \frac{az b^n}{m}\right)\right| &\le N \sum_{\substack{1\le z \le m-1\\(z,m)\ge m^\delta}} \frac{1}{z}\le  N\sum_{\substack{d|m \\ d\ge m^\delta}} \sum_{\ell \le m/d}  \frac{1}{d\ell}\\
&\ll N \sum_{\substack{d|m\\d\ge m^\delta}} \frac{\log (m/d)}{d}  = N \sum_{\substack{d|m\\ d\le m^{1-\delta}}} \frac{1}{m} \cdot d \log d\\
&\le \frac{N}{m}\cdot m^{1-\delta} \log (m^{1-\delta}) \cdot d(m) \ll N m^{-\delta/2},
\end{align*}
where in the first line we took $z= d\ell$ with $d=(z,m)$; in the second line we replaced $d$ with $m/d$; and in the third line $d(m)$ represents the number of divisors of $m$, which is known to be $O(m^{\delta/4})$ for any choice of $\delta>0$ (see, for example, \cite{Sandor}).

For the big part of the sum, we note that the internal sum of each term $azb^n/m$ has a denominator in lowest terms that exceeds $m^{1-\delta}$. Thus we can apply Theorem \ref{thm:secondary} to each of these sums, presuming $m$ is large enough (depending in part on $\delta$). Therefore, the sum over all these $z$ is bounded in the following way:
\begin{align*}
\sum_{\substack{1\le z \le m-1\\(z,m)< m^\delta}} \frac{1}{z} \left| \sum_{n=0}^{N-1} e\left( \frac{az b^n}{m}\right)\right| &\le \sum_{\substack{1\le z \le m-1\\ (z,m)< m^{\delta}}} \frac{1}{z} \cdot N \exp\left( -c(\log \log m^{1-\delta})^{3/2}\right)\\
&\ll N \log m \cdot \exp\left( -c(\log \log m^{1-\delta})^{3/2}\right).
\end{align*}
By replacing $c$ with $c/2$ (and then relabeling this term as $c$), we can show that the above sums are bounded by
\[
O\left( N \exp\left( - c (\log \log m)^{3/2}\right)\right),
\]
which is of larger order than $Nm^{-\delta/2}$.

Therefore, placing these bounds on the big part and small part of \eqref{eq:Rbound} into \eqref{eq:firstRstep}, we obtain the desired result.
\end{proof}

\subsection{Construction of normal numbers}

Given a real number $x$ with specified base-$b$ expansion $x=0.a_1a_2a_3\dots$ and a string $s=[d_1,d_2,\dots, d_k]$ of base-$b$ digits, we let
\[
\mathcal{N}_{x,s}(N) = \#\{ 1\le n \le N: a_{n-1+i}=d_i, 1\le i \le k\},
\]
similar to the definition given in the last section.

We say that a number $x$ is normal in base $b$ if 
\[
\lim_{N\to \infty} \frac{\mathcal{N}_{x,s}(N)}{N} = \frac{1}{b^k},
\]
for all finite length base-$b$ strings $s$.
Since only irrational numbers can be normal, this removes any concerns over ambiguity in the base-$b$ expansion of $x$. 

It follows by standard results of ergodic theory that almost all real numbers are normal to a given base $b$; however, exhibiting such numbers is difficult, and no commonly used mathematical constant (such as $\pi$, $e$, or $\sqrt{2}$) is known to be normal for any base. In 2002, Bailey and Crandall wrote a landmark paper giving a very general construction of normal numbers, majorizing many earlier papers of Stoneham and Korobov, in the following result:

\begin{thm}[Theorem 4.8 in \cite{BC}] Let $b,c>1$ be relatively prime integers. 
Let $(m_k)_{k=1}^\infty$ and $(n_k)_{k=1}^\infty$ be sequences of strictly increasing positive integers, and let $\mu_k=m_k-m_{k-1}$ and $\nu_k= n_k-n_{k-1}$ for $k\in \mathbb{N}$ (with $m_0=n_0=0$). Suppose that 
\begin{enumerate}
\item $(\nu_k)_{k=1}^\infty$ is non-decreasing; and,
\item There exists a constant $\gamma>1/2$ such that for sufficiently large $k$, we have
\[
\frac{\mu_k}{c^{\gamma n_k}} \ge \frac{\mu_{k-1}}{c^{\gamma n_{k-1}}}.
\]
\end{enumerate}
Then the number
\[
\alpha_{b,c,m,n} = \sum_{k=1}^\infty \frac{1}{b^{m_k}c^{n_k}}
\]
is normal in base $b$.
\end{thm}

As an example of this result, they show that the number $\sum_{k=1}^\infty 3^{-k}2^{-2^k}$ is normal in base $2$. However, their results would not be sufficient to prove the same result if $3$ were to be replaced with $5$. Nonetheless, this represents the ``most natural-looking" constructions of normal numbers to date, given the similarities to expansions such as $\log 2 = \sum_{n=1}^\infty 1/n2^n$. Most other constructions, such as Champernowne's constant $0.1234567891011\dots$, are formed by concatenating digits and have less resemblence to well-known constants. (Some further interesting properties of Stoneham--Korobov-type constructions have also been studied by Wagner and subsequent authors \cite{jung1996remarks,kano1993rings,wagnerrings}.)

For as well-regarded and influential as Bailey and Crandall's paper has been, it is surprising that this particular result has not been directly improved in any way. We use the new results of this paper to extend this result in the following way.

\begin{thm}
Let $b>1$ be a positive integer. Let $P$ be a finite set of primes all coprime to $b$ and let $(c_k)_{k=1}^\infty$ be a strictly increasing set of positive integers such that all primes dividing any $c_k$ come from $P$ and such that $c_{k}|c_{k+1}$. Let $(m_k)_{k=1}^\infty$ be a strictly increasing positive sequence of integers with $\mu_k=m_k-m_{k-1}$.

Suppose that there exists a constant $\epsilon>0$ such that
\[
\lim_{k\to \infty} \frac{\exp( (1+\epsilon) \log c_k/ \log \log c_k)}{\mu_k} = 0.
\]

Then the number
\[
\alpha_{b,c,m}= \sum_{k=1}^\infty \frac{1}{c_k b^{m_k}} 
\]
is normal in base $b$.
\end{thm}

\begin{proof}
We will follow the proof of Bailey and Crandall as much as possible and will elide over details that they treat more fully. 

We shall make use of the discrepancy $D_N(x_n)$ of a sequence $(x_n)_{n=0}^\infty$ with $x_n \in [0,1)$ for all $n$. The discrepancy is given by
\[
D_N(x_n) = \sup_{0\le a < b < 1} \left| \frac{\#\{n< N: x_n\in (a,b)\}}{N} - (b-a)\right|.
\]
If a sequence $(x_n)_{n_0}^m$ is finite, then we call $D_{m+1}(x_n)$ the discrepancy of this sequence. 
This definition is important to our proof because a number $x$ is normal in base $b$ if and only if the sequence given by $x_n= \{b^n x\}$, $n\ge 0$, satisfies $\lim_{N\to \infty} D_N(x_n)=0$ \cite[Theorem 2.2.(12)]{BC}. Here $\{z\}$ represents the fractional part of $z$. In fact, we can let $x_n$ be any sequence such that the limit of $\{x_n- b^n x\}$, as $n$ goes to $\infty$, is $0$ \cite[Theorem 2.2.(10)]{BC}.

Therefore, we define our desired ancillary sequence $(x_n)_{n=0}^\infty$ in the following way. We let $x_0=x_1=\dots =x_{m_1-1}=0$. Then we let $x_{m_1}=\{1/c_1\}$ and for all $j>0$ such that $m_1+j < m_2$, we let $x_{m_1+j} = \{b^j x_{m_1}\}$.  Then we let 
\[
x_{m_2} =\left\{  b^{m_2-m_1} x_{m_1}+ \frac{1}{c_2} \right\}
\]
and define $x_{m_2+j}$ for $m_2+j< m_3$ similarly. This process then continues. Note that by the divisibility properties of the sequence $(c_k)_{k=1}^\infty$ and the coprimality of $b$ and $P$, we have that each $x_{m_k}$ can be written as $a_k/c_k$ with $a_k$ and $c_k$ coprime and each $x_{m_k+j}$ with $0\le j \le \mu_{k+1}-1$ can be written as $\{a_k b^j/ c_k\}$. (This latter point will be important in \eqref{eq:discrepsum}.)

By construction $x_n$ approaches $\{b^n \alpha_{b,c,m}\}$ as $n$ goes to $\infty$. Thus to prove the result, we need to show that $D_N(x_n)$ goes to $0$ as $N$ goes to infinity. 

Consider a large integer $N$ and decompose it as $N=\mu_1+\mu_2+\dots + \mu_K+J$ with $J\in [1,\mu_{K+1}]$. Consider the finite sequence $(x_n)_{n=0}^N$ and decompose it too into $K+1$ sub-sequences $ (x_n)_{n=0}^{m_1-1}, (x_n)_{n=m_1}^{m_2-1}, \dots, (x_n)_{n=m_K}^N$, with corresponding discrepancies $D_1 , D_2, \dots,\\  D_{K+1}$. 

Then, by Lemma 4.2 of \cite{BC}, we have that
\begin{equation}\label{eq:DNbreakdown}
D_N(x_n) \le \sum_{k=1}^K \frac{\mu_k}{N} D_k + \frac{J}{N} D_{K+1}.
\end{equation}
We use the Erd\H{o}s-Turan discrepancy bound \cite[Theorem 2.2.(9)]{BC} to bound these smaller discrepancies. This inequality states that
\begin{equation}\label{eq:discrepsum}
D_k \ll \frac{1}{M} + \sum_{h=1}^M \frac{1}{h}\left| \frac{1}{J'} \sum_{j=0}^{J'-1} e\left( h x_{m_{k-1}+j} \right) \right| =\frac{1}{M} + \sum_{h=1}^M \frac{1}{h}\left| \frac{1}{J'} \sum_{j=0}^{J'-1} e\left( \frac{h a_{k-1} b^j}{c_{k-1}} \right) \right|,
\end{equation}
where the implicit constant is uniform, $J'$ denotes the number of terms in the corresponding subsequence, and $M$ is any positive integer we like. We will choose $M$ to be $\lfloor \sqrt{c_k} \rfloor$. 

We desire to apply Theorem \ref{thm:secondary} to each of the exponential sums in \eqref{eq:discrepsum}. Since $(a_{k-1},c_{k-1})=1$, $(b,c_{k-1})=1$, and $h\le \lfloor \sqrt{c_{k-1}} \rfloor$, we have that the denominator of $ha_{k-1} b^j/c_{k-1}$ is, in lowest terms, at least $\sqrt{c_{k-1}}$. However, these estimates can only be truly applied if $J'$ is sufficiently large compared with $c_k$, in this case larger than $\exp((1+\epsilon/3)\log c_{k-1}/\log \log c_{k-1})$ for some fixed $\epsilon>0$ and provided $c_{k-1}$ itself is sufficiently large. Thus, to account for the case where $J'$ is too small, we will simply include $\exp((1+\epsilon/3)\log c_{k-1}/\log \log c_{k-1})$  as an additional, trivial upper bound.

Thus, provided that $k$ is sufficiently large (so that $c_k$ is sufficiently large), we have that
\begin{align*}
D_k &\ll \frac{1}{\sqrt{c_k} } + \sum_{h=1}^M \frac{1}{h}\left( \frac{1}{J'} \exp\left( \frac{(1+\epsilon/3)\log c_{k-1}}{\log \log c_{k-1}}\right) +  \exp\left( -c (\log \log c_{k-1})^{3/2} \right) \right)\\
&\ll   \left( \frac{1}{J'} \exp\left( \frac{(1+\epsilon/3)\log c_{k-1}}{\log \log c_{k-1}}\right) +  \exp\left( -c (\log \log c_{k-1})^{3/2} \right) \right)\log c_{k-1}  ,
\end{align*}
where $c>0$ is the constant from Theorem \ref{thm:secondary} (which only depends on $P$). By adjusting the power in the exponent, we can absorb the $\log c_k$ term for large $k$ as follows:
\begin{equation}\label{eq:DKbound}
D_k \ll \frac{1}{J'} \exp\left( \frac{(1+\epsilon/2)\log c_{k-1}}{\log \log c_{k-1}}\right) +  \exp\left( -\frac{c}{2} (\log \log c_{k-1})^{3/2} \right) 
\end{equation}

Since the bounds in \eqref{eq:DKbound} only apply once $k$ is large enough,  let us choose a large $k_0$ (which, we will emphasize, can be any sufficiently large constant), such that these bounds are true for all $k\ge k_0$. For all $k< k_0$ we bound $D_k$ by $1$ and so the corresponding sum $\sum_{k=1}^{k_0-1} \mu_k D_k/N = N_0/N$ where $N_0$ is the number of terms in the corresponding subsequences.

Therefore, by placing these estimates into \eqref{eq:DNbreakdown} we have
\begin{align*}
D_N(x_n) &\ll \frac{N_0}{N} + \frac{1}{N} \sum_{k=k_0}^K\left( \exp\left( \frac{(1+\epsilon/2)\log c_{k-1}}{\log \log c_{k-1}}\right)+\mu_k \exp\left( -\frac{c}{2} (\log \log c_{k-1})^{3/2} \right) \right) \\
&\qquad + \frac{1}{N} \left(\exp\left( \frac{(1+\epsilon/2)\log c_{K}}{\log \log c_{K}}\right)+ J \exp\left( -\frac{c}{2} (\log \log c_{K})^{3/2} \right) \right)
\end{align*}
Since $\exp(\log x/\log\log x)$ is a strictly increasing function and  $\exp(-c/2 (\log\log x)^{3/2})$ is a strictly decreasing function once $x$ is large enough, we may assume that $k_0$ is large enough so that
\[
D_N(x_n)  \ll \frac{N_0}{N} + \frac{K}{N} \exp\left( \frac{(1+\epsilon/2)\log c_{K}}{\log \log c_{K}}\right) + \exp\left( -\frac{c}{2} (\log \log c_{k_0-1})^{3/2} \right).
\]
Since the sequence $c_k$ has each term dividing the next and is strictly increasing, we have that $2^{k-1} \le c_k$, therefore $K \ll \log c_K$. Moreover $1/N \le 1/\mu_K$.  Thus,
\[
D_N(x_n)  \ll \frac{N_0}{N} + \frac{1}{\mu_K} \exp\left( \frac{(1+\epsilon)\log c_{K}}{\log \log c_{K}}\right) + \exp\left( -\frac{c}{2} (\log \log c_{k_0-1})^{3/2} \right).
\]
Now, by assumption, the second term is $o(1)$ as $N$ tends to infinity, and the first term and third term are smaller than any given $\delta>0$ provided $k_0$ and $N$ are sufficiently large. Thus $\lim_{N\to \infty} D_N(x) =0$ and the theorem is proved.
\end{proof}

\section{Further questions}

Many times, the study of Korobov-type exponential sums is closely tied to the study of exponential sums $\sum_{n=1}^N e( a b_n / m)$, where $b_n$ are a linear recurrent sequence.  Can the methods studied here be used to give insights in those cases as well?

\section{Acknowledgments}

The author acknowledges assistance from the Research and Training Group grant DMS-1344994 funded by the National Science Foundation.

\bibliographystyle{amsplain}

\end{document}